\documentclass[12pt]{amsart}
\newtheorem{theorem}{Theorem}[section]
\newtheorem{lemma}[theorem]{Lemma}
\newtheorem{definition}[theorem]{Definition}

\newtheorem{conjecture}[theorem]{Conjecture}
\newtheorem{obs}[theorem]{Observation}

\newtheorem{question}[theorem]{Question}

\newtheorem{setup}[theorem]{Set-up}
\newtheorem{conv}[theorem]{Convention}
\newtheorem{example}[theorem]{Example}

\begin{document}
\title[Rings for which  general linear forms are exact zero divisors]
{Rings for which general linear forms are exact zero divisors}

\author{Ayden Eddings and Adela Vraciu}
\address{Ayden Eddings\\ Department of Mathematics\\ University of Nebraska-Lincoln\\ Lincoln \\ NE 68588\\ U.S.A.} \email{aeddings2@huskers.unl.edu}
\address{Adela Vraciu\\Department of Mathematics\\University of South Carolina\\
Columbia\\ SC 29208\\ U.S.A.} \email{vraciu@math.sc.edu}


\maketitle
\begin{abstract}
We investigate the standard graded $k$-algebras over a field $k$ of characteristic zero for which general linear forms are exact zero divisors. We formulate a conjecture regarding the Hilbert function of such rings. We prove our conjecture in the case when the ring is a  quotient of a polynomial ring by a  monomial idea, and also in the case when the ideal is generated in degree 2 and all but one of the generators are monomials. 
\end{abstract}

\section{Introduction}

Let $k $ be an algebraically closed field of characteristic zero, $P=k[x_1, \ldots, x_n]$ the polynomial ring over $ k $ in $n$ variables, $I\subseteq P$ a homogeneous ideal, and $R=P/I$.
A ring $R$ as above is called a standard graded $k$-algebra. The $i$th graded piece of $R$, $R_i$, consists of the images of all homogeneous polynomials in $P$ of degree $i$.

\begin{conv}
In order to avoid cumbersome notation, we will use the same symbol $f$ to denote an element of the polynomial ring $P$ as well as its image in $R$. When both rings $P$ and $R$ are used, we will specify whether $f$ is used to denote an element of $P$ or an element of $R$.  Note that $f=0$ as an element of $R$ is equivalent to $f\in I$ as an element of $P$.
\end{conv}

\begin{definition}
The Hilbert function of $R$ is the numerical function $i \to H_R(i):=\mathrm{dim}_k  (R_i)$
\end{definition}

The study of Hilbert functions has been a driving force of many important developments in commutative algebra (see for example \cite{St}, \cite{go}, \cite{gr}, \cite{hp}).
One particularly important question is the following:
\begin{question}
What numerical sequences arise as Hilbert functions of standard graded rings? What numerical sequences arise as Hilbert functions of standard graded rings with certain additional properties (e.g. reduced, Cohen-Macalay, Gorenstein)?
\end{question}
The first part of the question has been answered by Macaulay in \cite{mc}.

The aim of this paper is to study Hilbert functions of standard graded rings with the property that  general  linear forms are exact zero divisors. We reivew the relevant definitions below:
\begin{definition}
Let $x, y$ be elements of $R$. We say that the pair $(x, y)$ is a pair of exact zero divisors if 
$$
\mathrm{Ann}_R(x)=(y) \ \ \ \ \mathrm{and\ \ \ Ann}_R(y)=(x)
$$
where $(x), (y)$ denote the ideals generated by $x$ and $y$ respectively.

Note that if $R$ is an Artinian ring, then one of the two conditions above implies the other.
\end{definition}
Exact zero divisors were introduced in \cite{HS}. The main motivation for introducing this concept is that it is ``the next best thing" to regular elements, in the sense that it allows one to connect the homological behavior of modules over a quotiont of $R$ to that of modules over $R$ (see also \cite{BCG}). Here, we focus on questions about the existence of such exact zero divisors. 

We will further assume that the elements $x, y$ are homogeneous. 
A necessary condition for the Hilbert function of a standard graded algebra that has exact zero divisors of fixed degrees $d_1$ and $d_2$ was established in \cite{KSV}.
The specific question we pursue here is about existence of {\bf general}  linear exact zero divisor. In order to explain the concept of general linear form,  we identify a linear form $\displaystyle \ell =\sum_{i=1}^n a_ix_i$ with the vector of coefficients $(a_1, \ldots, a_n ) $ in the affine space $k^n$. 
\begin{definition}
We say that $R$ has general linear forms that are exact zero divisors if there exists a nonempty Zariski open set $\mathcal U$ in $k^n$ such that  every linear form $\displaystyle \ell =\sum_{i=1}^n a_ix_i$ with $(a_1, \ldots, a_n) \in \mathcal U$ is part of a pair of exact zero divisors. 
\end{definition}
Exact zero divisors provide a special type of {\em totally reflexive modules}. More specifically, if $x$ is part of a pair of exact zero divisors, then $(x)$ and $R/(x)$ are totally reflexive modules (see \cite{CFH} for a survey of totally reflexive modules and related topics). Thus, the question of existence of exact zero divisors is closely related to the question of existence of non-free totally reflexive modules. 
For Gorenstein Artininan ring $R$, every $R$-module is totally reflexive. However, if the ring is not Gorenstein, then there are two possibilities: either there are infinitely many indecomposable non-isomorphic totally reflexive modules, or else the only totally reflexive modules are the free modules (this is the main result  in \cite{CPST}).  There is in general no known criterion for how to distinuguish between these two cases. 

When $R_{\ge 3}=0$, the following necessary condition for existence of non-free totally reflexive modules is given by Yoshino:
\begin{theorem}\label{Yo}\cite{Yo}

Assume that $R$ as above is not Gorenstein and has $R_{\ge 3}=0$. If there exist non-free totally reflexive $R$-modules (in particular, if there exist exact zero divisors in $R$), then the following conditions must be satisfied:
\begin{enumerate}
\item $\mathrm{dim}_k(R_2)=\mathrm{dim}_k(R)_1 -1$

\item The ideal $I$ must be generated by homogeneous polynomials of degree 2. 
\end{enumerate}
\end{theorem}

We point out that the necessary conditions in Yoshino's theorem are not sufficient. Some examples or rings that satisfy the conditions in  Theorem (\ref{Yo}) but do not  have non-free totally reflexive modules can be found in  in \cite{AV}.

The condition that $\mathrm{dim}_k(R_2)=\mathrm{dim}_k(R_1)-1$ equivalent to the condition that the number of generators of $I$ is equal to 
$$
N=\left(\begin{array}{c} n+1 \\ 2 \\ \end{array}\right) -n+1
$$
If we allow $f_1, \ldots, f_N$ to be general homogeneous polynomials of degree 2 (meaning that the vector obtained by putting all their coefficients together belongs to a Zariski open set in the corresponding affine space), it is shown in \cite{co1} that $R$ has  exact zero divisors (and therefore non-free totally reflexive modules).

Additionally, the following result from \cite{br} shows that the existence of one pair of exact zero divisors for a ring $R$ is a ring as in Yoshino's theorem (with $R_{\ge 3}=0$) implies that general linear forms are exact zero divisors.
\begin{theorem}\label{br}\cite{br}
Let $R$ be a $k$-algebra with $R_{\ge 3}=0$. If $R$ has a pair of exact zero divisors, then a general linear form is part of a pair of exact zero divisors. 
\end{theorem}

\section{Statement of the problem and preliminaries}

Our goal is to find a converse to the result in Theorem (\ref{br}). Specifically, we ask the following:
\begin{question}
If $R$ is a standard graded $k$-algebra with the property that  general linear forms  are exact zero divisors, what can be said about the Hilbert function of $R$?
\end{question}
Note that we are no longer assuming that $R_{\ge 3}=0$. We conjecture that a condition similar to the one in Theorem (\ref{Yo})  must hold. Specifically, we propose the following
\begin{conjecture}\label{conj}
Assume that $R$ is a standard graded algebra, and $d>0$ is such that for general linear forms $\ell $, there exists $Q$ of degree $d-1$ such that $(\ell, Q)$ is a pair of exact zero divisors. Then 
$$
\mathrm{dim}(R_d)=\mathrm{dim}(R_{d-1})-1
$$
\end{conjecture}
We note that if $I$ is generated by homogeneous polynomials of degree $d$, and $(\ell, Q)$ is a pair of exact zero divisors with $\ell $ a linear form, then the degree of $Q$ must be $d-1$ by Proposition 1.9  in \cite{KSV2}.

We will prove that the conjecture is true in the case when $I$ is generated by monomials of arbitrary degrees, and also when all but one of the generators of $I$ are monomials, and all the generators have degree 2. 

The following remark shows that it will be enought to consider the case $\mathrm{dim}(R_d) \ge \mathrm{dim}(R_{d-1})-1$.

\begin{obs}\label{obs} If $\mathrm{dim}(R_d) \le \mathrm{dim}(R_{d-1})-2$, then there are no pairs of exact zero divisor pairs $(\ell, Q)$ with $\ell \in R_1$ and $Q\in R_{d-1}$.
\end{obs}
\begin{proof}
For every linear form $\ell$,  the map given by multiplication by $\ell : R_{d-1} \to R_d$ has a kernel of dimension at least two, and therefore $\mathrm{Ann}_R(\ell )$ cannot be generated by a single element of degree $d-1$. 
\end{proof}

We recall the definition of the Weak Lefschetz Property (WLP), which will be used in the proof for the monomial case:
\begin{definition}
We say that a stadard graded algebra $R$ has the Weak Lefschetz Property (WLP) if for general linear forms $\ell $, and for every degree $i>0$, the multiplication by $\ell $ map $:R_{i-1} \to R_i$ has maximal rank (meaning that it is injective when $\mathrm{dim}(R_{i-1} ) \le  \mathrm{dim}(R_i)$, and surjective when $\mathrm{dim}(R_{i-1} ) \ge  \mathrm{dim}(R_i)$).
\end{definition}

It is known that WLP holds for many classes of rings (see \cite{MN} for a survey). Most notably, it holds when $R=P/I$ with $I=(x_1^{a_1}, \ldots, x_n^{a_n})$ for any $a_1, \ldots, a_n>0$ by \cite{St2}.

We note the following connection between the problem we study and WLP:
\begin{obs}
If $\mathrm{dim}_k(R_d)\ge \mathrm{dim}_k(R_{d-1})$ and general linear forms $\ell $ are part of a pair of exact zero divisors $(\ell, Q)$ with $\mathrm{deg}(Q)=d-1$, then the map given by multiplication by $\ell :R_{d-1} \to R_d$ is not injective, and therefore $R$ does not have the Weak Lefschetz Property.
\end{obs}

\section{The monomial case}

Assume, as in the statement of the Conjecture, that there is a nonempty Zariski open set $\mathcal U$ of linear forms such that each $\ell $ in $\mathcal U$  is part of a pair of exact zero divisors. Since a nonempty Zariski open set is dense, this implies that there exists $\ell =\sum_{i=1}^n a_ix_i\in \mathcal U$ with all $a_i\ne 0$. The change of variables $x_i \to a_i x_i$ is a ring automorphism of the polynomial ring $P$. If $I$ is generated by monomials, it sends elements of $I$ to elements of  $I$ and therefore it gives rise to a ring automorphism of $R=P/I$. This shows that $\ell =\sum_{i=1}^na_ix_i$ with all $a_i\ne 0$ is part of a pair of exact zero divisors if and only if $L:=\sum _{i=1}^n x_i$ is.

\begin{lemma}\label{lem1}
Let $R=P/I$ with $P=k[x_1, \ldots, x_n]$, and $I$ a monomial ideal. Let $L=\sum_{i=1}^n x_i\in R$.
Let $Q=\sum_{\mu } a_{\mu} \mu $ be an element of $R_{d-1}$, where the summation is over all the monomials $\mu $ of degree $d-1$, and $a_{\mu } \in k$ is the coefficient of $\mu $.

Assume that $LQ=0$ (as an element of $R$). If $\mu $ is a monomial of degree $d-1$ with $a_{\mu } \ne 0$ and $M$ is a monomial of degree $2d-1$ such that $\mu $ divides $M$, then $M\in I$.
\end{lemma}
\begin{proof}
Assume that $M\notin I$. We will prove that $a_{\mu }=0$ for every $\mu $ of degree $d-1$ which divides $M$. Define $J$ to be the ideal generated by all monomials in $P$ that do not divide $M$. If $M=x_1^{a_1}x_2^{a_2} \cdots x_n^{a_n}$, then $J=(x_1^{a_1+1}, \ldots, x_n^{a_n+1})$. 
The assumption that $M\notin I$ means that $M$ is not divisible by any of the monomial generators of $I$, and therefore $I\subseteq J$. Thus, $LQ=0$ as an element of $R$ (i.e. $LQ\in I$) implies $LQ\in J$.

Consider $\overline{P}:=P/J$, and let $\overline{L}, \overline{Q}$ be the images of $L$ and $Q$ respectively in $\overline{P}$.
We know that $\overline{P}$ has WLP by \cite{St2}.

The socle degree of $\overline{P}$ is $a_1+\cdots + a_n =\mathrm{deg}(M)=2d-1$. By \ \ \ \ \  ,
this implies that 
$$
\mathrm{dim}_k(\overline{P}_{d-1}) \le \mathrm{dim}_k(\overline{P}_d),
$$
and the WLP tells us that the map given by multiplication by $\overline{L} : \overline{P}_{d-1} \to \overline{P}_d$ is injective. Since $\overline{L}\overline{Q}=0$, this implies $\overline{Q}=0$. This means that for every $\mu \notin J$, we must have $a_{\mu } =0$.
\end{proof}
We are now ready to prove Conjecture (\ref{conj}) in the monomial case:
\begin{theorem}\label{mt}
Let $I\subset P=k[x_1,\ldots, x_n]$ be a monomial ideal, and $R=P/I$. If there is a nonempty Zariski open set $\mathcal U$ such that  for all $\ell \in \mathcal U$  there is a $Q\in R_{d-1}$ with $(\ell, Q)$= a pair of exact zero divisors, then $\mathrm{dim}_k(R_d)=\mathrm{dim}_k(R_{d-1})-1$.
\end{theorem}
\begin{proof}
The assumption implies that $L=\sum_{i=1}^n x_i$ is part of a pair of exact zero divisors $(L, Q)$, with $Q\in R_{d-1}$.
Assume by way of contradiction that $\mathrm{dim}_k(R_d)\ne \mathrm{dim}_k(R_{d-1})-1$. Observation (\ref{obs}) then shows that we must have $\mathrm{dim}_k(R_d)\ge \mathrm{dim}_k(R_{d-1})$.

Since $LQ=0$, Lemma~\ref{lem1} shows that for all monomials $\nu \in R_d$, we have $\nu Q=0$ (for each term of $Q=\sum_{\mu } a_{\mu } \mu $, we have either $a_{\mu}=0$ or $\mu \nu =0$ as an element of $R$). Therefore, $R_d Q=0$ and, as such, $R_d\subseteq \mathrm{Ann}_R(Q)$. The assumption that $(L, Q)$ is a pair of exact zero divisors now implies that $R_d\subseteq (L)$, which means that the multiplication by $L$ map $:R_{d-1} \to R_d$ is surjective. However, we also have $Q$ in the kernel of that map, so the dimension of the image is less than or equal to $\mathrm{dim}(R_{d-1}) -1$, and therefore strictly less than  $\mathrm{dim}(R_d)$, which contradicts surjectivity.
\end{proof}

We end this section with an example to illustrate that general linear exact zero divisors can occur in rings other than the ones with $R_{\ge 3} = 0$ (the case when $R_{\ge 3}=0$ was already known  by the results in \cite{br}).

\begin{example}
Let $\displaystyle R=\frac{k[x_1, \ldots, x_n]}{(x_1^d)+ (x_2, \ldots, x_n)^d}$, $\displaystyle L=\sum_{i=1}^n x_i$, and 
$\displaystyle Q=\sum_{i=0}^{d-1} (-1)^i \ell_0 ^i x_1^{d-1-i}$, where $\displaystyle \ell_0 =\sum_{i=2}^n x_i$. Then $(L, Q)$ is a pair of exact zero divisors. Consequently, for generic linear form $\ell $, $\ell $ is part of a pair of exact zero divisors. 

\end{example}
\begin{proof}
We have $LQ=x_1^d-\ell _0^d$, which is zero as an element of $R$.

Let $Q'=\sum_{i=0}^D q_i x_1^{D-i}$ be an arbitrary element of $k[x_1, \ldots, x_n]$, written as a polynomial in $x_1$ with coefficients $q_i \in k[x_2, \ldots, x_n]$. We want to show that if $LQ'=0$ as an element of $R$, then $Q'$ must be a multiple of $Q$ when viewed as an element of $R$.

Since $L=x_1+\ell_0$, we have $\displaystyle LQ'=\sum_{i=0}^{D}  (q_i+\ell _0 q_{i-1} ) x_1^{D+1-i}$, where $q_{-1}=0$. The requirement that $LQ'=0$ as an element of $R$ means that all the coefficients of $x_1^i$ for $i< d$ must belong to $(x_2, \ldots, x_n)^d$. Therefore, we have 
$q_D+\ell _0 q_{D-1} \equiv 0, q_{D-1}+\ell _0 q_{D_2} \equiv0, \ldots, q_{D-d+2} + \ell _0 q_{D-d+1} \equiv 0$ (mod $(x_2, \ldots, x_n)^d$). Therefore, $q_{D-d+i} \equiv (-1)^{i-1} \ell _0^{i-1} q_{D-d+1}$ (mod $(x_2, \ldots, x_n)^d$) for all $1\le i \le d$.
As an element of $R$, $x_1^{d+j}=0$ for all $j\ge 0$, and 
$$
Q'=\sum_{i=D-d+1}^{D} q_i x_1^{D-i}=q_{D-d+1} (\sum_{i=1}^{d}(-1)^{i-1}\ell_0^{i-1} x^{d-i})
$$
and therefore it is a multiple of $Q$ as desired.
\end{proof}
\section{The case of degree 2 and one non-monomial generator}

We aim to extend the result of Theorem (\ref{mt}) beyond the case of monomial ideals. More precisely, we consider an ideal $I$ in which all but one of the generators are monomials.
The following set up will be in effect in this section:
\begin{setup}\label{su}
Let $J$ be an ideal generated by monomials of degree 2, and let  $P=k[x_1, \ldots, x_n]$, and let $f_1\ne  f_2$ be monomials of degree 2. Let $I=J+(f_1+f_2)$. 
\end{setup}
We wish to prove that Conjecture (\ref{conj})  is true under these assumptions. This will follow as a corollary of the following result, which is  the analogue of  Lemma~\ref{lem1}:
\begin{lemma}\label{th1}
Let $R$ be as in Set-up~\ref{su}.
For $\displaystyle \ell =\sum_{i=1}^na_ix_i$ a general linear form, if  $Q\in R_1$ is such that $(\ell, Q)$ is a pair of exact zero divisors in $R$, then $QR_2=0$.
\end{lemma}
The proof of Conjecture (\ref{conj}) under the assumptions of Set-up~\ref{su} follows as a corollary of Lemma~\ref{th1}:
\begin{theorem}
Let $R=P/I$ with $I$ as in Theorem~\ref{th1}. Assume that $\mathrm{dim}(R_2)\ne n-1$. Then a general  linear form $\ell $ cannot be part of a pair of exact zero divisors.
\end{theorem}

\begin{proof}
Assume by way of contradiction that for a general linear form $\ell$, there exists $Q$ of degree 1 such that $(\ell, Q)$ is a pair of exact zero divisors. 
It follows from 
 Observation (\ref{obs}) that  $\mathrm{dim}(R_2)\ge n=\mathrm{dim}(R_1)$.

Lemma (\ref{th1}) gives  $QR_2=0$, which implies $R_2\subseteq \mathrm{Ann}_R(Q)=(\ell)$, so the multiplication by $\ell $ map from $R_1$ to $R_2$ must be  surjective. However, $Q\ne 0$ is in the kernel of this map, and therefore the dimension of the image is less that or equal to $\mathrm{dim}(R_1)-1$, which is strictly less than the dimension of $R_2$, contradicting surjectivity.
\end{proof}

We note some preliminary facts before starting the proof of Lemma (\ref{th1}).
\begin{lemma}\label{conca}
Let $J$, $I$ and $R$  be as in Set-up~\ref{su}. Let $I_1:=J+(f_1), I_2:=J+(f_2)$. If $\ell$ is a general linear form and $Q$ is a linear form such that $(\ell, Q)$ is a pair of exact zero divisors in $R$, then there exist linear forms $Q_1$ and $Q_2$ such that $\ell Q_1\in I_1$, $\ell Q_2\in I_2$, and $Q=Q_1+Q_2$.
\end{lemma}
\begin{proof}
Choosing a monomial order, the initial ideal $\mathrm{in}(I)$ in degree 2 is equal to either $I_1$ or $I_2$ in degree 2. We may assume with no loss of generality that $(I_1)_2=(\mathrm{in}(I))_2$.
Theorem 1.1 in \cite{conca} shows that 
\begin{equation}\label{c}
\mathrm{dim}\left(\frac{P}{I+(\ell )}\right)_2 \le \mathrm{dim}\left(\frac{P}{I_1+(\ell )}\right)_2 
\end{equation}
We have short exact sequences 
$$
0 \longrightarrow \left(\frac{P}{(I:\ell )}\right)_1 \longrightarrow \left(\frac{P}{I}\right)_2 \longrightarrow \left(\frac{P}{I+(\ell)}\right)_2\longrightarrow 0
$$ 
and 
$$
0 \longrightarrow \left(\frac{P}{(I:\ell)}\right)_1 \longrightarrow \left(\frac{P}{I_1}\right)_2 \longrightarrow \left(\frac{P}{I+(\ell)}\right)_2\longrightarrow 0
$$ 
where the rightmost map is the canonical projection, and the leftmost map is multiplication  by $\ell$.
Therefore, 
$$
\mathrm{dim}\left(\frac{P}{I+(\ell )}\right)_2 =\mathrm{dim}\left(\frac{P}{I}\right)_2-\mathrm{dim}\left(\frac{P}{(I:\ell )}\right)_1
$$
and
$$
\mathrm{dim}\left(\frac{P}{I_1+(\ell )}\right)_2 =\mathrm{dim}\left(\frac{P}{I_1}\right)_2-\mathrm{dim}\left(\frac{P}{(I_1:\ell )}\right)_1
$$
Since $I$ and $I_1$ have the same number of generators of degree 2, the dimensions of $(P/I)_2$ and $(P/I_1)_2$ are the same. Therefore, equation (\ref{c}) shows that 
$$
\mathrm{dim}\left(\frac{P}{(I_1:\ell )}\right)_1 \le \mathrm{dim}\left(\frac{P}{(I:\ell )}\right)_1
$$
The existence of a $Q$ as in the assumption means that 
the right hand side of the above inequality is strictly less than $n$, and therefore the left hand side is also strictly less than $n$, meaning that there exists a nonzero $Q_1\in (I_1:\ell )$ of degree 1. 

We have $\ell Q\equiv \alpha (f_1+f_2)$ (mod $J$) and $\ell Q_1\equiv  \alpha _1 f_1$ (mod $J$), with $\alpha, \alpha_1 \in k$. If $\alpha_1=0$, we have $\ell Q_1\in J\subseteq I$, and the assumption that $(\ell, Q)$ is a pair of exact zero divisors in $R$ implies that $Q_1$ is equal to a multiple of $Q$; therefore, we may replace $Q_1$ by $Q$ and take $Q_2=0$ for the desired conclusion. If $\alpha_1 \ne 0$, we may replace $Q_1$ by $(\alpha Q_1)/\alpha_1$, so that $\ell Q_1\equiv\alpha  f_1$ (mod $J$), and take $Q_2=Q-Q_1$. We have $\ell Q_2\equiv \alpha f - \alpha f_1 = \alpha f_2 $ (mod $J$), and the desired conclusion holds.

\end{proof}

\begin{lemma}\label{facts}
Let $J$ and $I$ be as in Set-up~\ref{su}. Let $I_1=J+(f_1), I_2=J+(f_2)$. Let $\ell $ be a general linear form, and  let $Q_1, Q_2\in P_1$ such that $\ell Q_1\in I_1$ and $\ell Q_2\in I_2$. 
For every variable $u\in \{x_1, \ldots, x_n\}$, let $c_{u}$ and $d_{u}$ denote the coefficients of $u $ in $Q_1$ and $Q_2$ respectively.
Let $M$ be a monomial of degree 2.
Then:

{\rm a.} Assume $c_u\ne 0$, $M$ is not a multiple of $f_1$ and $u$ does not divide $f_1$. Then $uM\in J$.

{\rm b.} Assume  $d_u \ne 0$, $M$ is not a multiple of $f_2$ and $u$ does not divide $f_2$. Then $uM\in J$.
\end{lemma}

\begin{proof}
{\rm a.} By  Lemma (\ref{lem1}) applied to $I_1$ in the role of $I$ and $uM$ in the role of $M$, the assumption $c_u\ne 0$ implies that $uM\in I_1$. We note that $L$ in the statement of Lemma (\ref{lem1}) can be replaced by a general linear form $\ell$, so Lemma (\ref{lem1}) can be applied in this context.

Since $I_1=J+(f_1)$ is a monomial ideal, $uM\in I_1$ means $uM\in J$ or $uM$ is a multiple of $f_1$. The latter is not possible given that $u$ does not divide $f_1$ and $M$ is not a multiple of $f_1$.

{\rm b.} Same proof as {\rm a.}, using $I_2$ instead of $I_1$.

\end{proof}

We are now ready to prove Lemma (\ref{th1}).

\begin{proof}
Let $I_1=J+(f_1), I_2=J+(f_2)$. Let $Q_1, Q_2$ be as in the conclusion of Lemma (\ref{conca}). 
Let $Q_1=\sum_u c_u u$, $Q_2=\sum_u d_u u$, so that $Q=\sum_u b_u u$ with $b_u = c_u + d_u$.

Fix $M$ a monomial of degree 2, say $M=uv$ ($u, v$ are variables $\in \{x_1, \ldots, x_n\}$).  We need to show that $MQ\in I$. We will use the following observation:
\begin{equation}\label{sum}
\mathrm{If} \ M'Q\in I\  \mathrm{for\ all} \  M'=uv' \ \mathrm{with} \ v'\ne v ,  \ \mathrm{then}\  MQ\in I.
\end{equation}
Indeed, we have 
$$
a_vMQ=u\ell Q-\sum_{v'\ne v} a_{v'}(uv')Q
$$
where $a_{v'}$ is the coefficient of $v'$ in $\ell$ for every variable $v'$. 
The conclusion follows because the assumption that $\ell $ is general means that we can assume $a_v\ne 0$.

We consider the following cases:

{\bf Case 1:} $\mathrm{gcd}(f_1, f_2)=1$.  Say $f_1=xy, f_2=st$ (we are allowing $x$ and $y$ to possibly be the same, and $s$ and $t$ to possibly be the same).

Let $M$ be a monomial of degree 2. We claim that: 

{\rm 1.} If $M\ne x^2, y^2$, then $Q_1M\in I$, and 

{\rm 2.} If $M\ne s^2, t^2$, then  $Q_2M\in I$. 

Assuming the claim, it follows that $MQ=MQ_1+MQ_2\in I$ for all $M\ne x^2, y^2, s^2, t^2$.
For $M=x^2$,  we can apply observation (\ref{sum}) with $x^2$ in the role of $M$ to obtain $x^2Q \in I$ (since we know $M'Q\in I$ for all $M'=xv'$ with $v'\ne x$) Similarly, we also obtain $y^2Q, s^2Q, t^2Q\in I$, and therefore $MQ\in I$ for all monomials $M$ of degree 2, which is the desired conclusion.

We prove the first part of the claim (the second part is similar). Assume $M\ne x^2, y^2$.

We will show that $uM\in I$ for all $u$ with $c_u\ne 0$. We consider the cases $M=xy$ and $M\ne xy$.

For $M=xy=f_1$ (with $x\ne y$), we have $uM=uf_1\equiv uf_2 \ (\mathrm{mod}\, I)$. Assume by way of contradiction that $uM\notin I$. Then $uf_2\notin I$. However, $c_u\ne 0$ implies $uf_2\in I_1$ by Lemma~\ref{lem1} applied to $I_1$ in the role of $I$ and $uf_2$ in the role of $M$. Since $I_1=J+(f_1)$ is a monomial ideal, $uf_2\in I_1$ implies $uf_2\in J$ or $uf_2$ is a multiple of $f_1$. 
But $uf_2\in J$ implies $uf_2\in I$, since $J\subseteq I$, and $uf_2$ equal to a multiple of $f_1$ implies $\mathrm{gcd}(f_1, f_2)\ne 1$. We get a contradiction in either case.

 Now consider $M\ne xy$. For $u\ne x, y$ we have $uM\in J$  from Lemma~\ref{facts} ({\rm a}).
We will be done once we show that $xM\notin I$ implies $c_x=0$, and $yM\notin I$ implies $c_y=0$. We only prove the first part (the second part is similar).

 Note that $xM\notin I$ implies $xM\notin J$.
If $M$ is not a multiple of $y$, then $xM$ cannot be a multiple of $f_1$, and therefore $xM\notin I_1$. Lemma~\ref{lem1} applied to $I_1$ in the role of $I$ shows $c_x=0$, as desired.  Consider the case when $M$ is a multiple of $y$, say $M=yz$ with $z\ne y$ (since $M\ne y^2$). The assumption $xM\notin I$ implies $xM=xyz\notin J$, and therefore $xz\notin J$. Since $xz$ is not a multiple of $f_1$, this shows $xz\notin I_1$. The assumption $LQ_1\in I_1$ (keeping in mind that $I_1$ is a monomial ideal) means that for each monomial $uv$ of degree 2 which is not in $I$, the coefficient of $uv$ in $LQ$ must be equal to zero. Therefore, we have $c_x+c_z=0$. Thus  $c_x=0\Leftrightarrow c_z=0$.
If $c_z\ne 0$, Lemma~(\ref{lem1}) applied to $I_1$ in the role of $I$ gives $f_2z\in I_1$. Since $\mathrm{gcd}(f_1, f_2)=1$, $f_2z$ cannot be a multiple of $f_1$, and therefore $f_2z\in J$.
But $f_2z\equiv f_1z=xM$ (mod $I$), so $xM\notin I$ implies $f_2z\notin I$, and therefore $f_2z\notin J$. This shows that $c_z$ (and therefore also $c_x$) must be zero.

{\bf Case 2:} $\mathrm{gcd}(f_1, f_2)=x$, and both $f_1$ and $f_2$ are not squares.  Say $f_1=xy, f_2=xt$ with $x, y, t$  all distinct.

We claim that for $M$=monomial of degree 2, if $M\ne x^2, y^2, t^2$ then $Q_1M, Q_2M\in I$, and therefore $QM=Q_1M+Q_2M\in I$.  Once we prove the claim, it will also  follow that $QM\in I$ for $M=x^2, y^2, t^2$ using observation (\ref{sum}).

Let $M\ne x^2, y^2, t^2$. We prove $Q_1M\in I$ ($Q_2M\in I$ is similar) by proving that for all $u$ with $c_u\ne 0$, we have $uM\in I$. We consider the cases $M=f_1$ and $M\ne f_1$.

Let $M=xy=f_1$. Assume by way of contradiction that $c_u\ne 0$ and $uM\notin I$. Since $uf_1\equiv -uf_2$ (mod $I$), it follows that $uf_2\notin I$, and therefore $uf_2\notin J$. However, $c_u\ne 0$ implies $uf_2\in I_1$ by Lemma~\ref{lem1} applied to $I_1$ in the role of $I$.  Therefore, $uf_2$ must be a multiple of $f_1$. This implies $u=y$, so $uM=xy^2$. Note that $xy^2\notin I$ implies $y^2\notin J$, and since $y^2$ is not a multiple of $f_1$, we must also have $y^2\notin I_1$. The assumption $LQ\in I_1$ now implies that the coefficient of $y^2$ in $LQ$ must be 0. This coefficient is $c_y$, so we contradict the assumption that $c_u=0$.

Let $M\ne f_1$. Lemma~\ref{facts} part a. implies that $uM\in I$ for all $u\ne x, y$ with $c_u\ne 0$. It remains to show that if $c_x\ne 0$, then $xM\in I$, and if $c_y\ne 0$, then $yM\in I$.

 Assume $c_x\ne 0$ and assume by way of contradiction that $xM\notin I$. By Lemma~\ref{lem1} applied to $I_1$ in the role of $I$, we have $xM\in I_1$. The only way this is possible is if $xM$ is a multiple of $f_1$, and therefore $M$ is a multiple of $y$.
Say $M=uy$ with $u\ne x, y$ (since we already handles the case $M=xy$, and we are assuming $M\ne y^2$). 
So we have $xM=uxy\notin J$, which implies $ux\notin J$ and thus $ux\notin I_1$. The coefficient of $ux$ in $LQ$ is $c_u+c_x$, so $LQ\in I_1$ implies $c_u+c_x=0$. Since $c_x\ne 0$, we also have $c_u\ne 0$. Lemma~\ref{lem1} applied to $I_1$ in the role of $I$ now implies $uf_2\in I_1$.
We have $uf_2\equiv uf_1=xM$ (mod $I$), so $xM\notin I$ implies $uf_2\notin I$, and therefore $uf_2\notin J$. Thus, $uf_2\in I_1$ implies that $uf_2$ must be a multiple of $f_1$. Therefore we must have $u=y$, and $M=y^2$, which contradicts our assumption.

Assume $c_y\ne 0$ and assume by way of contradiction that  $yM\notin I$.  By Lemma~\ref{lem1} applied to $I_1$ in the role of $I$, we have $yM\in I_1$, and therefore $yM$ must be a multiple of $f_1$, which means that $M$ must be a multiple of $x$.
 Say $M=ux$, with $u\ne x$ (since we are assuming $M\ne x^2$). Since $yM=yux\notin I$, it follows that $uy\notin J$, and since $uy$ is not a multiple of $f_1$, $uy\notin I_1$. The coefficient of $uy$ in $LQ$ is $c_u+c_y$, so $LQ\in I_1$ implies $c_u+c_y=0$.  Since $c_y\ne 0$, we also have $c_u\ne 0$.
Lemma~\ref{lem1} applied to $I_1$ in the role of $I$ implies $uf_2\in I_1$. We have $uf_2\equiv uf_1=yM$ (mod $I$), so the assumption $yM\notin I$ implies $uf_2\notin I$, and therefore $uf_2\notin J$. It follows that $uf_2$ must be a multiple of $f_1$, and so $u=y$ and $M=xy=f_1$. We have already proved earlier that $yf_1\in I$.

{\bf Case 3:} $f_1, f_2$ have a variable in common, and one of them is a square. Say $f_1=x^2, f_2=xy$ (with $x\ne y$).

We  claim that for every monomial $M$ of degree 2, if $M\ne x^2, y^2$, then  $MQ_1, MQ_2\in I$ (and therefore $MQ=MQ_1+MQ_2\in I$).

Once we prove the claim, we can use observation (\ref{sum}) for $M=x^2$ and $M=y^2$, and it will follow that $MQ\in I$ for all $M\in P_2$.

We begin by proving the first part of the claim ($MQ_1\in I$ for $M\ne x^2, y^2$) by proving that $uM\in I$ for every variable $u$ with $c_u\ne 0$.

 If $u\ne x$, Lemma~(\ref{facts}) part a. implies $uM\in J$, and therefore $uM\in I$. It remains to consider $u=x$.
 Assume by way of contradiction $c_x\ne 0$ and $xM\notin I$.  By Lemma~\ref{lem1} applied to $I_1$ in the role of $I$, we have $xM\in I_1$, and therefore $xM$ must be a multiple of $f_1=x^2$. This means that $M$ is a multiple of $x$. Say that $M=xw$ with $w\ne x$. The assumption $xM\notin I$ implies $xw\notin J$, and since $xw$ is not a multiple of $f_1$, $xw\notin I_1$. The coefficient of $xw$ in $LQ_1$ is $c_x+c_w$, and the assumption that $LQ_1\in I_1$ implies $c_x+c_w=0$. Since we are assuming $c_x\ne 0$, $c_w$ is also not zero.

 We have $xM=f_1w\equiv f_2w \ (\mathrm{mod} \ I)$. So $xM\notin I$ implies $f_2w\notin J$. On the other hand, since  $c_w\ne 0$, Lemma~\ref{lem1} implies that $f_2w\in I_1$, and therefore $f_2w$ must be a multiple of $f_1$. This is not possible (since $w, y\ne x$).

Now prove the second part of the claim ($MQ_2\in I$ for all $M\ne x^2, y^2$) by proving that $uM\in I$ for every variable $u$ with $d_u\ne 0$.

 For $u\ne x, y$ with $d_u\ne 0$, we know that $uM\in J$ by Lemma~\ref{facts} part b. It remains to consider $u=x$ and $u=y$.

For $u=x$: assume  by way of contradiction $d_x\ne 0$ and $xM\notin I$. By Lemma ~\ref{lem1} applied to $I_2$ in the role of $I$, we have $xM\in I_2$, and therefore $xM$ must be a multiple of $f_2$. This implies that  $M$ must be a multiple of $y$. Say $M=yw$ with $w\ne y$. The assumption $xM=xyw\notin I$ implies $xw\notin J$, and since $xw$ is not a multiple of $f_2(=xy)$, we also  have $xw\notin I_2$. The coefficient of $xw$ in $LQ_2$ is $d_x+d_w$, and the assumption that $LQ_2\in I_2$ implies $d_x+d_w=0$. Therefore $d_x\ne 0$ implies $d_w\ne 0$. We have $xM=wf_2\equiv wf_1 \ (\mathrm{mod} \ I)$, so $wf_1\notin J$; however, $d_w\ne 0$ implies $wf_1\in I_2$, so $wf_1$ must be a multiple of $f_2$. This is not possible (since $w\ne y$).

For $u=y$: assume  by way of contradiction $d_y\ne 0$ and $yM\notin I$. By Lemma~\ref{lem1} applied to $I_2$ in the role of $I$, we have $yM\in I_2$, and therefore $yM$ must be a multiple of $f_2$. This implies that  $M$ must be a multiple of $x$. Say $M=xw$, with $w\ne x$. Assume by contradiction $yM=xyw\notin I$. Then $yw\notin J$, and since $yw$ is not a multiple of $f_2$, we have $yw\notin I_2$. The coefficient of $yw$ in $L_2Q$ is $d_y+d_w$, and the assumption $LQ_2\in I_2$ implies $d_y+d_w=0$. Therefore $d_y\ne 0$ implies $d_w\ne 0$. We have $yM=wf_2\equiv wf_1 \ (\mathrm{mod} \ I)$, so $wf_1\notin J$; however, $d_w\ne 0$ implies $wf_1\in I_2$, so $wf_1$ must be a multiple of $f_2$. The only way that this is possible is if $w=y$. But if $w=y$ we have $yM=xy^2$ and the assumption $yM\notin I$ implies $y^2\notin J$. Since $y^2$ is not a multiple of $xy$, we also have $y^2\notin I_2$. 
The coefficient of $y^2$ in $LQ_2$ is $d_y$, so the assumption that $L_2Q\in I_2$ implies $d_y=0$. This is a contradiction.

\end{proof}

\end{document}